\documentclass[12pt,a4paper]{article}

\usepackage{amsfonts}
\usepackage{amsmath}
\usepackage{amsbsy}
\usepackage{amsxtra}
\usepackage{latexsym}
\usepackage{amssymb}
\usepackage{cite}
\begin{document}
\title{How to project onto the monotone nonnegative cone using
Pool Adjacent Violators type algorithms
\thanks{{\it 1991 A M S Subject Classification.} Primary 90C33;
Secondary 15A48, {\it Key words and phrases.} Metric projection  onto the monotone nonnegative cone}}
\author{A. B. N\'emeth\\Faculty of Mathematics and Computer Science\\Babe\c s Bolyai University, Str. Kog\u alniceanu nr. 1-3\\RO-400084 Cluj-Napoca, Romania\\email: nemab@math.ubbcluj.ro \and S. Z. N\'emeth\\School of Mathematics, The University of Birmingham\\The Watson Building, Edgbaston\\Birmingham B15 2TT, United Kingdom\\email: nemeths@for.mat.bham.ac.uk}
\date{}
\maketitle

\begin{abstract}

The metric projection onto an order nonnegative cone from the metric projection onto the
corresponding order cone is derived. Particularly, we can use Pool Adjacent Violators-type 
algorithms developed for projecting onto the monotone cone for projecting onto the monotone 
nonnegative cone too.

\end{abstract}

\newenvironment{proof}{{\bf Proof.}}{\hfill$\Box$\\}
\newtheorem{theorem}{Theorem}
\newcommand{\R}{\mathbb R}
\newcommand{\N}{\mathbb N}
\newcommand{\lang}{\langle}
\newcommand{\rang}{\rangle}
\newcommand{\pe}{\preceq}

\newtheorem{corollary}{Corollary}

\section{Introduction}
The metric projection onto convex cones is an important tool in solving problems in
metric geometry, statistics, image reconstruction etc. In almost all applications the 
projection onto a convex cone is part of an iterative process, hence its efficiency is of 
crucial importance. 

Projecting onto an order cone \cite{FernandezRuedaSalvador1998} is a fundamental tool for 
solving isotonic regression problems (see \cite{RobertsonWrightDykstra1988}). We will call an 
order nonnegative cone the intersection of an order cone with
the nonnegative orthant. An order nonnegative cone correspond to an isotonic 
nonnegative regression problem which we will define from an isotonic regression problem by 
superimposing the nonnegativity of variables.

An order nonnegative cone is a pointed cone which is the subcone of the corresponding order
cone which is not pointed.

A special case of the isotone regression problem is the case of the regression with
respect to a complete order (\cite{RobertsonWrightDykstra1988}, \cite{BestChakravarti1990}).
The corresponding cone is called monotone cone. It  turns out
that the corresponding monotone nonnegative cone is important in the metric  geometry  and 
the image reconstruction. Whereby the importance of projecting onto this cone.

A
simple finite method of projection onto the so called isotone projection cones
(cones having the property that the metric projection onto them
is isotone with respect to the order relation defined by these cones) proposed by us (see
\cite{NemethNemeth2009}) has become important in the effective handling of
all the problems involving projection onto these cones.
The monotone nonnegative convex cone used in the Euclidean
distance geometry (see \cite{Dattorro2005}) is an isotone projection one. 
Our method has become important in the effective handling of 
the problem of map-making from relative distance information e.g., stellar cartography 
(see 
{\small
\begin{verbatim}
www.convexoptimization.com/wikimization/index.php/Projection_on_Polyhedral_Cone 
\end{verbatim}}

\noindent and Section 5.13.2.4 in \cite{Dattorro2005}).

Due to the importance for the regression theory of the projection onto the monotone cone, 
there are various efficient methods of projection onto it.
These methods emerge from the so called Pool Adjacent Violators (PAV) algorithm,
which has nothing to do with the geometric approach of \cite{NemethNemeth2009}.
The PAV algorithm exploits the specific feature of the monotone cone, and due to its efficiency
it would be desirable its adaptation for the monotone nonnegative cone too.
Our note aims to do this by joining PAV and the geometric approach
specific for \cite{NemethNemeth2009}.

In this note we show that the projection of a point onto an order nonnegative cone is the
positive part (with respect to the lattice structure of the nonnegative orthant in $\R^m$) of the 
projection of the point onto the corresponding order cone.

\section{Preliminaries}

Let $W$ be a \emph{convex cone} in $\R^m$, i. e., a nonempty set with
(i) $W+W\subset W$ and (ii) $tW\subset W,\;\forall \;t\in \R_+ =[0,+\infty)$.

The cone $K\subset \R^m$ is said \emph{pointed}, if $K\cap (-K)=\{0\}.$

The \emph{polar} of the convex cone $W$ is the set
$$W^{\perp}:=\{y\in \R^m:\;\lang x,y\rang \leq 0,\;\forall \;x\in K\},$$
where $\lang\cdot,\cdot\rang $ is a scalar product in $\R^m$.

If $W$ is a closed convex cone, then 
from the extended Farkas lemma (or bipolar theorem, see e.g. Theorem 14.1 in \cite{Rockafellar1970} p. 121) is
$$ W^{\perp \perp}=(W^\perp)^\perp = W.$$
If $Z$ is another closed convex cone, then $W$ and $Z$ are called
\emph{mutually polar} if $Z=W^\perp$ (and hence $W=Z^\perp$ by the lemma of Farkas).

The scalar product $\lang\cdot,\cdot\rang$ defines a metric $d$ on $\R^m$ by setting 
$d(x,y)=\lang x-y,x-y\rang^{1/2}$ for any $x,y\in\R^m$.
Denote by $P_W:\R^m\to W$ the \emph{projection} onto the closed convex cone $W$ (or the nearest point mapping), which
associates to $x\in \R^m$ its (unique with respect to the metric defined by the scalar product
$\lang\cdot,\cdot\rang$ \cite{Zarantonello1971}) nearest point $P_Wx\in W$ in $W$.

The projection mapping $P_W$ onto $W$ is 
characterized by the following theorem of Moreau \cite{Moreau1962}.
\vspace{2mm}

\noindent
\textbf{Theorem [Moreau]}\hspace{2mm}
\textit{Let $W,Z\subset \R^m$ be two mutually 
polar convex cones in $H$. Then, the following statements are equivalent:
\begin{enumerate}
	\item[(i)] $z=x+y,~x\in W,~y\in Z$ and $\langle x,y\rangle=0$,
	\item[(ii)] $x=P_Wz$ and $y=P_Zz$.
\end{enumerate}}

\section{Order cones and order nonnegative cones}

Suppose that $\R^m$ is endowed with a Cartesian coordinate system,
and $x\in \R^m$, $x=(x^1,\dots,x^m)$, where $x^i$ are the coordinates of
$x$ with respect to this reference system. 

Endow the index set $\{1,\dots,m\}$ with a partial order $\pe$. The \emph{order cone} 
\cite{FernandezRuedaSalvador1998} with respect to the partial order $\pe$ is defined by

\begin{equation*}
	W_\pe=\{x\in \R^m: x^i\le x^j\textrm{ whenever }i\pe j\}.
\end{equation*}

The \emph{order nonnegative cone} corresponding to the order cone $W_\pe$ 
is defined by

\begin{equation*}
	K_\pe=W_\pe\cap\R^m_+.
\end{equation*}

We obviously have 

\begin{equation}\label{posconsubscon}
K_\pe\subset W_\pe.
\end{equation}


\section{Projection onto an order nonnegative cone via the projection onto the corresponding 
order cone}

Let $b=(b^1,\dots,b^m)\in\R^m$ and $w=(w_1,\dots,w_m)\in\R^m_{++}$ a vector of positive
weights.
The partial order $\pe$ defines a directed acyclic graph over the nodes $\{1,\dots,m\}$ such
that $(i,j)\in E_\pe$ whenever $i\pe j$, where $E_\pe$ is the set of edges of the graph. Recall
that the
\emph{isotonic regression problem} is the following minimization problem:
\[
\begin{array}{ll}
	\textrm{Minimize}   & \sum_{i=1}^n w_i (x_i - b_i)^2\\
	\textrm{subject to} & x_i\le x_j,~\forall (i,j)\in E_\pe.
\end{array}
\]
We define the \emph{isotonic nonnegative regression} problem as the following 
minimization problem:
\[
\begin{array}{ll}
	\textrm{Minimize}   & \sum_{i=1}^n w_i (x_i - b_i)^2\\
	\textrm{subject to} & x_i\le x_j,~\forall (i,j)\in E_\pe,\\
			    & x_k\ge 0\textrm{ for }k=1,\dots,m.	
\end{array}
\]
If we define the scalar product $\langle\cdot,\cdot\rangle$ by 
$\langle x,y\rangle=w_1x_1y_1+\dots+w_mx_my_m$ for any $x,y\in\R^m$, then the isotonic 
regression problem is equivalent to projecting $b$ onto $W_\pe$. Similarly, the isotonic 
nonnegative regression problem is equivalent to projecting $b$ onto $K_\pe$. However, the 
following results hold for any scalar product $\langle\cdot,\cdot\rangle$ in $\R^m$.

For any vector $a=(a^1,\dots,a^m)\in\R^m$ denote by $a^+=\sup\{a,0\}$ its positive part with 
respect to the lattice structure of $\R^m_+$. Thus, if $m=1$, then $a^+=\max\{a,0\}$ and if 
$m>1$, then $a^+=(a^{1+},\dots,a^{m+})$, where $a^{i+}=(a^i)^+$ for any $i\in\{1,\dots,m\}$. 
Similarly denote by $a^-=\sup\{-a,0\}$ its negative part
with respect to the lattice structure of $\R^m_+$. Thus, if $m=1$, then $a^-=\max\{-a,0\}$ and 
if $m>1$, then $a^-=(a^{1-},\dots,a^{m-})$, where $a^{i-}=(a^i)^-$ for any $i\in\{1,\dots,m\}$.
It is easy to see that $a=a^+-a^-$.

Using the notation in the preceding section we have the 
following result: 

\begin{theorem}\label{fo}
For any $y\in \R^m$ we have $$P_{K_\pe}y=(P_{W_\pe}y)^+.$$
\end{theorem}

\begin{proof} 
We have from Moreau's theorem that
\begin{equation}\label{mormor}
	y=P_{W_\pe}y+P_{W_\pe^\perp}y,\;\;\lang P_{W_\pe}y,P_{W_\pe^\perp}y\rang =0.
\end{equation}
Denote $v=(P_{W_\pe}y)^+$. Using the notations in the theorem, let $u=P_{W_\pe}y-v,$ that is,
$$u=-(P_{W_\pe}y)^-.$$

Denote $z=P_{W_\pe^\perp}y$.
Then, (\ref{mormor}) becomes
\begin{equation}\label{mormor1}
	y=u+v+z \;\textrm{with}\;\lang u+v,z\rang =0.
\end{equation}

First we show that from the special form of $W_\pe$, we have $u,v\in W_\pe$. Indeed, if we 
denote $x=P_{W_\pe}y$, then \[u=(-x^{1-},\dots,-x^{m-})\] and 
\[v=(x^{1+},\dots,x^{m+}).\] It is easy to see that the functions $\R\ni t\mapsto -t^-$ and 
$\R\ni t\mapsto t^+$ are monotone increasing. Since $x\in W_\pe$ we have $x^i\le x^j$ whenever
$i\pe j$. Hence, by using the monotonicity of $\R\ni t\mapsto -t^-$, we get 
$u^i\le u^j$ whenever $i\pe j$. Similarly, by using the monotonicity of $\R\ni t\mapsto t^+$, 
we get $v^i\le v^j$ whenever $i\pe j$. Thus, $u,v\in W_\pe$ which together with 
(\ref{mormor1}) and $z\in W_\pe^{\perp}$ yield 
\begin{equation}\label{ketvmer}
	\lang u,z\rang =\lang v,z\rang=0.
\end{equation}

From $K_\pe\subset W_\pe$ (see (\ref{posconsubscon})), it follows that
\begin{equation}\label{antitonic} 
	W_\pe^{\perp}\subset K_\pe^{\perp}.
\end{equation}

From the fact that all the coordinates of $u$ are nonpositive
and the elements in $K_\pe$ have nonnegative coordinates, it follows
that 
\begin{equation}\label{v2ikm}
	u\in K_\pe^{\perp}.
\end{equation}
 
Let us write now 
$$y=v+(u+z)$$
and observe that $v\in K_\pe.$ Further $u+z\in K^{\perp}$
from (\ref{antitonic}) and (\ref{v2ikm}).
We also have 
$$\lang v,u+z\rang =\lang v,u\rang +\lang v,z\rang =0,$$
because of (\ref{ketvmer}) and the forms of $u$ and $v$.

Using again the theorem of Moreau, it
follows the conclusion of the theorem.

\end{proof}

\section{Projection onto the monotone nonnegative cone
via the projection onto the monotone cone}

There are various efficient methods for projecting onto the monotone cone
emerging from the PAV algorithm (see e. g. \cite{BestChakravarti1990}).
These methods are intimately
related to the special structure of the monotone cone, and their
justification is more function theoretic than geometric.
Due to their simplicity their usage is desirable for projection onto
the monotone nonnegative cone too.

As a consequence of Theorem \ref{fo} we have the following corollary; 

\begin{corollary}
Suppose that $W$ is the monotone cone, that is,
\begin{equation*}
W=\{x\in \R^m:\;x^1\leq x^2\leq ...\leq x^m\},
\end{equation*}
and 
\begin{equation*}
K=\{x\in \R^m:\;0\leq x^1\leq x^2\leq ...\leq x^m\} 
\end{equation*}
is the monotone nonnegative cone.
Then for an arbitrary $y\in \R^m$ it holds
$$P_Ky=(P_Wy)^+,$$ 
where $^+$ stands for the lattice operation defined
by the order induced by the nonnegative orthant in $\R^m.$
\end{corollary}

To exploit the efficiency of PAV-type algorithms in projecting onto
the monotone nonnegative cone, we can proceed as follows:
For an arbitrary $y\in \R^m$ we
can determine $P_Wy=(x^1,...,x^m)$, the projection of $y$ on the monotone cone $W$,  by
using a PAV-type algorithm (e.g. the algorithm in \cite{deLeeuwHornikMair2009}).
Then, take the vector
$$v=(x^{1+},...,x^{m+}),$$
where $x^{i+}$ denotes the ``positive part'' of the coordinate $x^i$.
Then, according the above corollary,
the projection of $y$ on the monotone nonnegative cone $K\subset W$
engendering $W$ is given by   
$$P_Ky=v.$$


\end{document}